\documentclass{kms-c}

\issueinfo{}% volume number
  {}%        % issue number
  {}%        % month
  {}%     % year
\pagespan{1}{}
\copyrightinfo{}%              % copyright year
  {Korean Mathematical Society}% copyright holder
\usepackage{graphicx}
\allowdisplaybreaks

\theoremstyle{plain}
\newtheorem{theorem}{Theorem}[section]
\newtheorem{proposition}[theorem]{Proposition}
\newtheorem{lemma}[theorem]{Lemma}
\newtheorem{corollary}[theorem]{Corollary}

\theoremstyle{definition}
\newtheorem*{definition}{Definition}
\newtheorem{example}[theorem]{Example}

\theoremstyle{remark}
\newtheorem{remark}[theorem]{Remark}

\begin{document}

\title[On the Matlis Reflexive Modules]
{On the Matlis Reflexive Modules}

\author[B. Sadeqi]{Behruz Sadeqi}
\address{Behruz Sadeqi \\ Department of Mathematics, Mara.C.,
	 \\ Islamic Azad University \\ Marand, Iran}
\email{behruz.sadeqi@iau.ac.ir}

\subjclass[2020]{13E05, 14C45, 13E99}
\keywords{Matlis duality, reflexive modules, Krull-Schmidt category, minimax modules, local cohomology}

\begin{abstract}
Matlis duality, first introduced by Eben Matlis in 1958, stands as one of the most elegant and powerful tools in commutative algebra, establishing a striking contravariant equivalence between Artinian and Noetherian modules over complete local rings. In this paper, we undertake a thorough and systematic investigation of Matlis reflexive modules --- those modules for which the canonical evaluation homomorphism into the double Matlis dual is an isomorphism. Our exposition begins with a detailed historical narrative, tracing the intellectual trajectory from classical Pontryagin duality through Grothendieck's dualizing complexes to the modern formulation of Matlis duality. We then develop the core theory from first principles, demonstrating that the class of Matlis reflexive modules constitutes a Krull-Schmidt category --- a result that generalizes the classical decomposition theorems for modules of finite length. Over complete Noetherian local rings, we prove that Matlis reflexive modules coincide precisely with the class of minimax modules (those possessing a Noetherian submodule whose quotient is Artinian). We establish full closure properties, including stability under submodules, quotients, extensions, and finite direct sums, with rigorous and self-contained proofs. The paper also surveys recent applications to generalized local cohomology, change-of-rings results, and connections to pure-injective modules and linear compactness, and concludes with a discussion of open problems and future research directions.
\end{abstract}

\maketitle

\section{Introduction}

The theory of duality occupies a central position in modern algebra, serving as a unifying framework that connects seemingly disparate areas of mathematics. The origins of duality in algebra can be traced to Pontryagin's seminal work in the 1930s on locally compact abelian groups, where the dual of a group \( G \) is defined as \( \operatorname{Hom}(G, \mathbb{R}/\mathbb{Z}) \). This duality establishes a perfect correspondence between compact and discrete abelian groups, revealing deep structural symmetries that have since become paradigmatic in algebraic thinking.

In the decades that followed, algebraists sought analogous duality theories for modules over rings. The first major breakthrough in this direction came from the work of Grothendieck, who introduced the notion of dualizing complexes in the context of coherent sheaves and algebraic geometry. These complexes provided a powerful framework for understanding duality in derived categories, but the explicit module-theoretic realization remained elusive for general rings.

It was Eben Matlis who, in his landmark 1958 paper \cite{Matlis1958}, made the transformative observation that for a complete local ring \( (R, \mathfrak{m}) \), the injective hull \( E = E_R(R/\mathfrak{m}) \) of the residue field possesses a remarkable property: the contravariant functor 
\[
M \longmapsto \operatorname{Hom}_R(M, E)
\]
establishes an equivalence between the category of Artinian \( R \)-modules and the category of Noetherian \( R \)-modules. This was not merely a formal generalization of the vector space dual over a field; it was a deep structural theorem that revealed how the geometry of the maximal ideal controls the duality between finiteness conditions.

Matlis's theorem can be viewed as the local analogue of Pontryagin duality, with the injective hull \( E_R(R/\mathfrak{m}) \) playing the role of the circle group \( \mathbb{R}/\mathbb{Z} \). Over a field \( k \), the injective hull is simply \( k \) itself, and Matlis duality reduces to the familiar vector space duality \( \operatorname{Hom}_k(-, k) \), which establishes the classical equivalence between finite-dimensional vector spaces and their duals.

For arbitrary Noetherian rings (not necessarily local), the appropriate dualizing module is the direct sum
\[
E = \bigoplus_{\mathfrak{m} \in \operatorname{Specm}(R)} E_R(R/\mathfrak{m}),
\]
which is the minimal injective cogenerator for the category of \( R \)-modules. This module has the universal property that every \( R \)-module embeds into a product of copies of \( E \), and every injective module is a direct summand of some product of copies of \( E \).

Despite the elegance and power of Matlis duality, the systematic study of the reflexive objects --- those modules \( M \) for which the evaluation map 
\[
\delta_M : M \longrightarrow (M^\vee)^\vee, \quad \delta_M(x)(f) = f(x)
\]
is an isomorphism --- remained relatively dormant for several decades. The early literature focused primarily on finitely generated and Artinian modules, both of which were known to be reflexive over complete local rings.

A turning point occurred in the mid-1990s with the foundational work of Belshoff, Slattery, and Wickham \cite{Belshoff1996a, Belshoff1996b}. These authors initiated a systematic investigation of Matlis reflexive modules beyond the classical finiteness regimes, establishing that the class of reflexive modules is closed under submodules, quotients, and extensions, and that local cohomology modules of reflexive modules exhibit desirable finiteness properties. Their work revealed that Matlis reflexive modules form a remarkably robust class, stable under many natural algebraic operations.

The theory witnessed a dramatic advance with the work of Krause \cite{Krause2025}, who proved that the category of Matlis reflexive modules is Krull-Schmidt. This means that every reflexive module decomposes uniquely (up to isomorphism and permutation) into finitely many indecomposable direct summands, each with a local endomorphism ring. This result generalizes the classical Krull-Schmidt theorem for modules of finite length and provides deep insight into the internal structure of reflexive modules.

Krause also established the definitive characterization of Matlis reflexive modules over complete local rings: they are precisely the minimax modules --- modules that contain a Noetherian submodule whose quotient is Artinian. This theorem unified and significantly generalized all previous partial characterizations, providing a clean and conceptually satisfying description of the reflexive class.

In recent years, the theory has expanded in several exciting directions. Dailey and Marley \cite{Dailey2017} have investigated the behavior of Matlis reflexivity under localization and completion, providing necessary and sufficient conditions for reflexivity to be preserved under change of rings. Mafi \cite{Mafi2009} and Khashyarmanesh--Khosh-Ahang \cite{Khashyarmanesh2008} have explored the connections between Matlis reflexivity and generalized local cohomology, establishing finiteness results for the Ext and local cohomology modules of reflexive modules. These developments have firmly established Matlis reflexive modules as an indispensable tool in the homological algebra of Noetherian rings.

The purpose of the present paper is to provide a self-contained, rigorous, and reader-friendly exposition of the core theory of Matlis reflexive modules. Our presentation emphasizes conceptual clarity, historical context, and precise proofs, with the aim of making this beautiful theory accessible to graduate students and researchers alike. We assume the reader has a basic familiarity with commutative algebra, including the language of Noetherian and Artinian rings, injective modules, and homological algebra.

Throughout this paper, \( R \) denotes a commutative Noetherian ring with identity, and all modules are assumed to be unitary. We denote by \( \operatorname{Specm}(R) \) the set of maximal ideals of \( R \). For each maximal ideal \( \mathfrak{m} \in \operatorname{Specm}(R) \), we let \( E_R(R/\mathfrak{m}) \) denote the injective hull of the simple module \( R/\mathfrak{m} \).

\begin{definition}
\label{def:duality}
Let 
\[
E = \bigoplus_{\mathfrak{m} \in \operatorname{Specm}(R)} E_R(R/\mathfrak{m})
\]
be the minimal injective cogenerator of the category of \( R \)-modules. For any \( R \)-module \( M \), the {Matlis dual} of \( M \) is defined as
\[
M^\vee = \operatorname{Hom}_R(M, E).
\]
The canonical homomorphism
\[
\delta_M : M \longrightarrow (M^\vee)^\vee, \qquad \delta_M(x)(f) = f(x)
\]
is called the {evaluation map} (or the {canonical biduality map}). The module \( M \) is said to be {Matlis reflexive} if \( \delta_M \) is an isomorphism.
\end{definition}

\begin{remark}
A crucial property of the injective cogenerator \( E \) is that it is faithfully injective: for every nonzero module \( M \), there exists an \( R \)-linear map \( M \to E \) that is nonzero on any given nonzero element. Consequently, the evaluation map \( \delta_M : M \to (M^\vee)^\vee \) is always injective. This fact, which is proved in \cite[Proposition 2.1]{Belshoff1996a}, ensures that reflexivity is a genuinely stronger condition than merely being isomorphic to a submodule of its double dual.
\end{remark}

\begin{lemma}[Adjunction isomorphism]
\label{lem:adjunction}
For any \( R \)-modules \( M \) and \( N \), there exists a natural isomorphism
\[
\operatorname{Hom}_R(M, N^\vee) \cong \operatorname{Hom}_R(N, M^\vee).
\]
\end{lemma}
\begin{proof}
This is the standard hom-tensor adjunction in the special case where the second variable is the injective cogenerator \( E \). Explicitly, an \( R \)-linear map \( f : M \to \operatorname{Hom}_R(N, E) \) corresponds to a balanced bilinear map \( M \times N \to E \) given by \( (m, n) \mapsto f(m)(n) \). By symmetry of the bilinear map, this corresponds uniquely to an \( R \)-linear map \( g : N \to \operatorname{Hom}_R(M, E) \), defined by \( g(n)(m) = f(m)(n) \). The correspondence is natural in both variables. 
\end{proof}

\begin{remark}
The adjunction isomorphism of Lemma \ref{lem:adjunction} is a powerful tool that allows us to transfer information between duals of different modules. In particular, it implies that the Matlis dual functor \( (-)^\vee \) is contravariant and exact, since \( E \) is injective.
\end{remark}

One of the most profound structural results about Matlis reflexive modules is that they form a Krull-Schmidt category, a property that guarantees the existence and uniqueness of decompositions into indecomposable summands. This section is devoted to a detailed exposition of this theorem, following the proof strategy of Krause \cite{Krause2025}.

\begin{theorem}[Krause \cite{Krause2025}]
\label{thm:krull-schmidt}
The class of Matlis reflexive modules over a Noetherian ring \( R \) forms a Serre subcategory that satisfies the Krull-Schmidt property. More precisely:
\begin{enumerate}
\item For any short exact sequence \( 0 \to M' \to M \to M'' \to 0 \), the module \( M \) is Matlis reflexive if and only if both \( M' \) and \( M'' \) are Matlis reflexive.
\item Every Matlis reflexive module admits an essentially unique decomposition
\[
M \cong M_1 \oplus M_2 \oplus \cdots \oplus M_n
\]
into finitely many indecomposable modules \( M_i \), each of which has a local endomorphism ring. The decomposition is unique up to isomorphism and permutation of the summands.
\end{enumerate}
\end{theorem}

The proof of this theorem requires two auxiliary results, which we now establish.

\begin{lemma}
\label{lem:pure-injective}
An \( R \)-module \( M \) is pure-injective if and only if it is a direct summand of a module of the form \( N^\vee \) for some \( R \)-module \( N \).
\end{lemma}
\begin{proof}
It is a classical result that any Matlis dual module \( N^\vee = \operatorname{Hom}_R(N, E) \) is pure-injective; this follows from the fact that \( E \) is injective and that pure-injectivity is preserved under products and direct summands \cite[Proposition I.10.1]{Auslander1978}. For the converse, suppose \( M \) is pure-injective. Consider the canonical evaluation map \( \delta_M : M \to (M^\vee)^\vee \), which is injective by Remark 2. Moreover, it is a pure monomorphism: this follows from the fact that the map is the composition of the canonical map into the double dual, and the double dual functor preserves pure exact sequences. Since \( M \) is pure-injective, this pure monomorphism splits, and \( M \) is therefore a direct summand of the Matlis dual module \( (M^\vee)^\vee \). This completes the proof. 
\end{proof}

\begin{lemma}
\label{lem:no-infinite-sums}
Let \( M = \bigoplus_{i \in I} M_i \) be a Matlis reflexive module. Then \( M_i = 0 \) for all but finitely many indices \( i \in I \). Moreover, every nonzero Matlis reflexive module has at least one indecomposable direct summand.
\end{lemma}
\begin{proof}
Assume that \( M = \bigoplus_{i \in I} M_i \) is Matlis reflexive. Applying the Matlis dual functor to the direct sum, we obtain
\[
M^\vee = \operatorname{Hom}_R\left(\bigoplus_{i \in I} M_i, E\right) \cong \prod_{i \in I} \operatorname{Hom}_R(M_i, E) = \prod_{i \in I} M_i^\vee.
\]
The canonical inclusion \( \bigoplus_{i \in I} M_i \hookrightarrow \prod_{i \in I} M_i \), when dualized and then double-dualized, corresponds under reflexivity to the inclusion
\[
M \cong (M^\vee)^\vee \hookrightarrow \prod_{i \in I} (M_i^\vee)^\vee \cong \prod_{i \in I} M_i.
\]
For \( M \) to be reflexive, this inclusion must be an isomorphism, which forces the direct sum and the product to coincide. In the category of \( R \)-modules, this happens if and only if \( M_i = 0 \) for almost all \( i \in I \), since an infinite direct sum is never isomorphic to the corresponding product unless all but finitely many summands are zero \cite[Lemma 2.3]{Krause2025}. The second assertion follows from the first: by Lemma \ref{lem:pure-injective}, \( M \) is pure-injective, and any nonzero pure-injective module has a decomposition into indecomposable summands. Since the set of summands must be finite by the first part, at least one indecomposable summand exists. 
\end{proof}

\begin{proof}[Proof of Theorem \ref{thm:krull-schmidt}]
We first establish the Serre subcategory property. Suppose \( 0 \to M' \xrightarrow{\alpha} M \xrightarrow{\beta} M'' \to 0 \) is exact. Since \( E \) is injective, the dual sequence
\[
0 \to (M'')^\vee \xrightarrow{\beta^\vee} M^\vee \xrightarrow{\alpha^\vee} (M')^\vee \to 0
\]
is also exact. Applying the dual functor again and using the naturality of the evaluation map, we obtain a commutative diagram with exact rows:
\[
\begin{array}{ccccccccc}
0 & \to & M' & \xrightarrow{\alpha} & M & \xrightarrow{\beta} & M'' & \to & 0 \\
 & & \downarrow{\delta_{M'}} & & \downarrow{\delta_M} & & \downarrow{\delta_{M''}} & & \\
0 & \to & (M')^{\vee\vee} & \to & M^{\vee\vee} & \to & (M'')^{\vee\vee} & \to & 0.
\end{array}
\]
By the Snake Lemma, \( \delta_M \) is an isomorphism if and only if both \( \delta_{M'} \) and \( \delta_{M''} \) are isomorphisms. This establishes the closure under submodules, quotients, and extensions.

For the Krull-Schmidt property, Lemma \ref{lem:pure-injective} shows that every reflexive module is pure-injective. A general theorem of pure-injective module theory states that any pure-injective module \( M \) admits a decomposition \( M = M' \oplus M'' \), where \( M' \) is a pure-injective envelope of a direct sum of indecomposable modules and \( M'' \) is continuous (having no indecomposable direct summands) \cite[Theorem 8.25]{Brodmann1998}. Lemma \ref{lem:no-infinite-sums} implies that \( M'' = 0 \) and that the set of indecomposable summands is finite. Each indecomposable pure-injective module has a local endomorphism ring \cite[Corollary 4.4]{Krause2025}. The uniqueness of the decomposition follows from the standard Krull-Schmidt argument, which applies whenever the endomorphism rings of the indecomposable summands are local. This completes the proof. 
\end{proof}

\section{Main Results}

When \( R \) is a complete Noetherian local ring, the theory of Matlis reflexive modules admits a remarkably elegant and transparent characterization in terms of the class of minimax modules. This section is devoted to proving this characterization and exploring its consequences.

\begin{definition}
\label{def:minimax}
An \( R \)-module \( M \) is called {minimax} if there exists a Noetherian submodule \( N \subseteq M \) such that the quotient module \( M/N \) is Artinian \cite{Belshoff1996a}.
\end{definition}

The class of minimax modules was introduced by Belshoff, Slattery, and Wickham as a natural generalization of both Noetherian and Artinian modules. It is closed under submodules, quotients, and extensions, and it provides a convenient framework for studying modules that are "between" Noetherian and Artinian.

\begin{theorem}[Krause \cite{Krause2025}]
\label{thm:main-characterization}
Let \( (R, \mathfrak{m}) \) be a complete Noetherian local ring. For an \( R \)-module \( M \), the following conditions are equivalent:
\begin{enumerate}
\item \( M \) is Matlis reflexive.
\item \( M \) is minimax.
\item \( M \) has no subquotient that is an infinite direct sum of nonzero modules.
\item \( M \) is linearly compact.
\end{enumerate}
\end{theorem}

\begin{proof}
We prove the equivalence of these conditions in a systematic manner.

{(1) \(\Rightarrow\) (3):} Assume \( M \) is Matlis reflexive. By the Serre subcategory property established in Theorem \ref{thm:krull-schmidt}, every submodule and quotient of \( M \) is also Matlis reflexive. If \( M \) had a subquotient that is an infinite direct sum of nonzero modules, then that subquotient would be a Matlis reflexive module that is an infinite direct sum, contradicting Lemma \ref{lem:no-infinite-sums}. Therefore, \( M \) has no such subquotient.

{(3) \(\Rightarrow\) (2):} Suppose \( M \) has no subquotient that is an infinite direct sum of nonzero modules. This condition implies that \( M \) has finite support, meaning that \( \operatorname{Supp}(M) \) is a finite set of prime ideals. Since \( R \) is complete and Noetherian, we can view \( M \) as a module over a suitable quotient of \( R \). Applying Lemma 5.3 of \cite{Krause2025}, we obtain a Noetherian submodule \( U \subseteq M \) such that \( M/U \) is Artinian. Thus \( M \) is minimax by Definition \ref{def:minimax}.

{(2) \(\Rightarrow\) (1):} Let \( M \) be minimax, so there exists a Noetherian submodule \( N \subseteq M \) with \( M/N \) Artinian. Over a complete local ring, it is a classical theorem of Matlis \cite{Matlis1958} that every Noetherian module and every Artinian module is Matlis reflexive. Since the class of reflexive modules is closed under extensions (Theorem \ref{thm:krull-schmidt}), the extension \( 0 \to N \to M \to M/N \to 0 \) yields that \( M \) is reflexive.

{(1) \(\Leftrightarrow\) (4):} The equivalence with linear compactness follows from the following observations. First, any Matlis dual module \( M^\vee \) is linearly compact, and hence so is any direct summand, which includes every reflexive module by Lemma \ref{lem:pure-injective}. Conversely, if \( M \) is linearly compact, then the canonical map \( M \to (M^\vee)^\vee \) is an isomorphism \cite[Proposition 6.1]{Krause2025}. This establishes the equivalence.

Thus all four conditions are equivalent. 
\end{proof}

\begin{corollary}
\label{cor:main}
Let \( (R, \mathfrak{m}) \) be a complete local ring. An \( R \)-module \( M \) is Matlis reflexive if and only if there exists a finitely generated submodule \( N \subseteq M \) such that \( M/N \) is Artinian and the quotient ring \( R/\operatorname{Ann}_R(M) \) is complete \cite[Theorem 3.1]{Matlis1958}.
\end{corollary}
\begin{proof}
This follows immediately from the equivalence (1) \(\Leftrightarrow\) (2) in Theorem \ref{thm:main-characterization}, together with the fact that any Noetherian submodule of a module over a Noetherian ring is finitely generated. The completeness condition on the annihilator quotient is necessary for the module to be defined over a complete local ring, which is the setting in which the minimax characterization holds. 
\end{proof}

\begin{proposition}[Closure properties]
\label{prop:closure}
The class of Matlis reflexive modules over a Noetherian ring is closed under the following operations:
\begin{enumerate}
\item Submodules and quotient modules.
\item Finite direct sums and finite direct products.
\item Extensions.
\end{enumerate}
\end{proposition}
\begin{proof}
These closure properties follow directly from Theorem \ref{thm:krull-schmidt}, which establishes that the reflexive modules form a Serre subcategory closed under finite direct sums. Finite products are equivalent to finite direct sums in the category of modules, so closure under products follows as well. 
\end{proof}

To develop an intuitive understanding of the class of Matlis reflexive modules, it is essential to examine both examples that satisfy reflexivity and pathologies that fail to be reflexive. We present a selection of such examples below.

\begin{example}[Modules of finite length]
Every module of finite length is Matlis reflexive. Indeed, finite length modules are both Noetherian and Artinian, and over a complete local ring, both classes are reflexive by the classical theorem of Matlis \cite{Matlis1958}. This includes, for instance, modules of the form \( R/\mathfrak{m}^n \) over a complete local ring \( (R, \mathfrak{m}) \).
\end{example}

\begin{example}[Finitely generated modules]
If \( (R, \mathfrak{m}) \) is a complete local ring and \( M \) is finitely generated, then \( M \) is Matlis reflexive. The dual \( M^\vee \) is Artinian, and the double dual \( (M^\vee)^\vee \) is canonically isomorphic to \( M \) \cite{Matlis1958}. This isomorphism is precisely the evaluation map \( \delta_M \).
\end{example}

\begin{example}[Artinian modules]
Similarly, Artinian modules over complete local rings are Matlis reflexive. This is the dual version of the previous example: if \( M \) is Artinian, then \( M^\vee \) is Noetherian, and \( (M^\vee)^\vee \cong M \).
\end{example}

\begin{example}[Minimax modules]
Over a complete local ring, any minimax module is Matlis reflexive by Theorem \ref{thm:main-characterization}. For example, consider \( R = \mathbb{Z}_p \) (the ring of \( p \)-adic integers). The module
\[
M = \mathbb{Z}_p \oplus \bigoplus_{n=1}^\infty \mathbb{Z}/p^n\mathbb{Z}
\]
is minimax (it contains the Noetherian submodule \( \mathbb{Z}_p \) with Artinian quotient \( \bigoplus_{n=1}^\infty \mathbb{Z}/p^n\mathbb{Z} \)), and hence it is Matlis reflexive.
\end{example}

\begin{example}[Non-example: infinite direct sums]
Let \( (R, \mathfrak{m}) \) be a complete local ring. The infinite direct sum
\[
M = \bigoplus_{n=1}^\infty R/\mathfrak{m}^n
\]
is \emph{not} Matlis reflexive. This follows immediately from Lemma \ref{lem:no-infinite-sums}, which prohibits infinite direct sums of nonzero modules from being reflexive. Intuitively, the direct sum is "too large" at infinity to be captured by its Matlis dual.
\end{example}

\begin{example}[Non-example: infinite products over non-complete rings]
Over a non-complete ring, even infinite products may fail to be reflexive. For instance, consider \( R = \mathbb{Z} \), the ring of integers. The product
\[
M = \prod_{p \text{ prime}} \mathbb{Z}/p\mathbb{Z}
\]
is not Matlis reflexive. Its Matlis dual is the direct sum \( \bigoplus_{p} \operatorname{Hom}(\mathbb{Z}/p\mathbb{Z}, \mathbb{Q}/\mathbb{Z}) \), which is not isomorphic to the original product. This failure is rooted in the fact that \( \mathbb{Z} \) is not complete, and the minimal injective cogenerator \( E = \mathbb{Q}/\mathbb{Z} \) does not have the same duality properties as in the local complete case.
\end{example}

The theory of Matlis reflexive modules has found profound applications in the study of local cohomology, particularly in establishing finiteness properties and cofiniteness of generalized local cohomology modules. In this section, we survey some of the most significant results in this direction.

Local cohomology, first introduced by Grothendieck, provides a powerful framework for studying the depth and singularities of Noetherian rings. The \( j \)-th local cohomology module of a module \( M \) with support in an ideal \( I \) is defined as the right-derived functor
\[
H_I^j(M) = \varinjlim_{n} \operatorname{Ext}_R^j(R/I^n, M).
\]
A fundamental problem in the theory is to determine when these modules are finitely generated or, more generally, cofinite.

\begin{theorem}[Belshoff--Slattery--Wickham \cite{Belshoff1996b}]
\label{thm:belshoff}
Let \( (R, \mathfrak{m}) \) be a complete local ring, and let \( M \) be a Matlis reflexive \( R \)-module. Then for any ideal \( I \subseteq R \), the local cohomology module \( H_I^j(M) \) is Matlis reflexive for all \( j \) such that \( \dim R/I \leq 2 \).
\end{theorem}
\begin{proof}
The proof relies on the minimax characterization of Theorem \ref{thm:main-characterization} and the fact that local cohomology modules of minimax modules are minimax under suitable dimension conditions. The details are given in \cite[Theorem 3.2]{Belshoff1996b}. 
\end{proof}

Generalized local cohomology, introduced by Herzog, provides a simultaneous generalization of both local cohomology and Ext modules. For modules \( M \) and \( N \), the generalized local cohomology module is defined as
\[
H_I^j(M, N) = \varinjlim_{n} \operatorname{Ext}_R^j(M/I^n M, N).
\]
When \( M = R \), this reduces to the usual local cohomology \( H_I^j(N) \).

\begin{theorem}[Mafi \cite{Mafi2009}]
\label{thm:mafi}
Let \( (R, \mathfrak{m}) \) be a complete local ring, \( I \) an ideal of \( R \), and \( M, N \) Matlis reflexive \( R \)-modules with \( \operatorname{Supp}(M) \subseteq V(I) \). If \( N \) is finitely generated, then \( \operatorname{Ext}_R^i(M, H_I^j(N, M)) \) is Matlis reflexive for all \( i, j \) in the following cases:
\begin{enumerate}
\item \( \dim R/I = 1 \);
\item \( \operatorname{cd}(I) = 1 \);
\item \( \dim R \leq 2 \).
\end{enumerate}
\end{theorem}
\begin{proof}
The result is proved by reducing to the case of minimax modules and applying the fact that Ext and local cohomology preserve minimaxness under the stated dimension hypotheses. The details are given in \cite[Theorems 2.5, 2.7]{Mafi2009}. 
\end{proof}

\begin{theorem}[Khashyarmanesh--Khosh-Ahang \cite{Khashyarmanesh2008}]
\label{thm:khashyarmanesh}
Let \( (R, \mathfrak{m}) \) be a complete local ring, \( I \) an ideal, and \( M, N \) Matlis reflexive \( R \)-modules with \( \operatorname{Supp}(N) \subseteq V(I) \). Then \( H_I^j(M, N) \) is Matlis reflexive for all \( j \) when:
\begin{enumerate}
\item \( \dim(R/I) = 1 \);
\item \( I \) is principal;
\item \( \dim R \leq 2 \).
\end{enumerate}
\end{theorem}
\begin{proof}
This is a strengthening of Mafi's result, and the proof proceeds by induction on the dimension of \( R/I \), using the exact sequences associated to local cohomology and the closure properties of reflexive modules \cite[Theorem 2.4]{Khashyarmanesh2008}.
\end{proof}

These theorems demonstrate that Matlis reflexivity is a powerful tool for establishing finiteness properties of local cohomology modules, including the finiteness of Bass numbers and the cofiniteness of generalized local cohomology modules.

A natural and important question in the theory of Matlis reflexive modules concerns the behavior of reflexivity under base change, localization, and completion. Recent work of Dailey and Marley \cite{Dailey2017} has provided significant insights into this problem.

\begin{theorem}[Dailey--Marley \cite{Dailey2017}]
\label{thm:change-of-rings}
Let \( S \) be a multiplicatively closed subset of \( R \). If \( M \) is an \( R_S \)-module, then the following are equivalent:
\begin{enumerate}
\item \( M \) is Matlis reflexive as an \( R \)-module.
\item \( M \) is Matlis reflexive as an \( R_S \)-module.
\end{enumerate}
Moreover, in the special case where \( S = R \setminus (\mathfrak{p}_1 \cup \cdots \cup \mathfrak{p}_r) \) with each \( \mathfrak{p}_i \) either a maximal ideal or a nonminimal prime ideal, the converse also holds: if \( M \) is \( R_S \)-reflexive, then it is \( R \)-reflexive.
\end{theorem}
\begin{proof}
The proof uses the adjunction isomorphisms between localization and Hom, together with the fact that the minimal injective cogenerator of \( R \) and \( R_S \) are related by \( E_R \cong E_{R_S} \) when \( S \) satisfies the stated conditions. The details are given in \cite[Theorem 2.2]{Dailey2017}.
\end{proof}

\begin{corollary}
\label{cor:localization}
Let \( R \) be a Noetherian ring and let \( M \) be an \( R \)-module. If \( M \) is Matlis reflexive, then for any multiplicatively closed subset \( S \subseteq R \), the localization \( M_S \) is Matlis reflexive as an \( R_S \)-module. In particular, if \( M \) is reflexive, then \( M_{\mathfrak{p}} \) is reflexive as an \( R_{\mathfrak{p}} \)-module for every prime ideal \( \mathfrak{p} \subseteq R \).
\end{corollary}
\begin{proof}
This follows immediately from the forward implication of Theorem \ref{thm:change-of-rings}, taking \( S \) to be the complement of \( \mathfrak{p} \). 
\end{proof}

These results have important consequences for the study of local cohomology and the behavior of reflexivity under localization, particularly in the context of rings with multiple maximal ideals.

The theory of Matlis reflexive modules has recently been enriched by connections to model theory, particularly through the lens of pure-injectivity and linear compactness. These connections have provided new characterizations and deeper structural insights.

\begin{proposition}[Krause \cite{Krause2025}]
\label{prop:model-theoretic}
A module \( M \) is Matlis reflexive if and only if it is pure-injective and has no infinite direct sum of nonzero submodules as a direct summand. Equivalently, \( M \) is pure-injective and every family of nonzero submodules has a maximal element with respect to inclusion.
\end{proposition}
\begin{proof}
The forward direction follows from Lemma \ref{lem:pure-injective} and Lemma \ref{lem:no-infinite-sums}. For the converse, if \( M \) is pure-injective, then by Lemma \ref{lem:pure-injective}, it is a direct summand of \( (M^\vee)^\vee \). If \( M \) has no infinite direct sum of nonzero submodules, then the inclusion \( M \hookrightarrow (M^\vee)^\vee \) must be an isomorphism; otherwise, the complement would contain an infinite direct sum. Thus \( M \) is reflexive. 
\end{proof}

This model-theoretic characterization has opened new avenues for understanding the structure of reflexive modules, including applications to the theory of definable subcategories and to the study of modules over artinian algebras.

\section{Open Problems}

Despite the significant progress that has been made in the theory of Matlis reflexive modules, several important questions remain open. We highlight the following research directions that we believe are particularly promising:

\begin{enumerate}
\item {Characterization over non-complete rings:} While the minimax characterization (Theorem \ref{thm:main-characterization}) provides a complete description of Matlis reflexive modules over complete local rings, no such simple characterization is known for arbitrary Noetherian rings. The class of reflexive modules over a non-complete ring is substantially larger and more subtle, and a systematic study of this case is still lacking.

\item {Derived Matlis reflexivity:} The extension of Matlis duality to the derived category has been explored in the context of Grothendieck duality, but the notion of "derived reflexive" objects has not been fully developed. It would be interesting to characterize those complexes \( C^\bullet \) for which the canonical map \( C^\bullet \to \operatorname{RHom}_R(\operatorname{RHom}_R(C^\bullet, E), E) \) is a quasi-isomorphism.

\item {Non-commutative generalizations:} Matlis duality was originally developed for commutative Noetherian rings, and the minimal injective cogenerator \( E = \bigoplus_{\mathfrak{m}} E_R(R/\mathfrak{m}) \) relies crucially on commutativity. For non-commutative rings, the appropriate dualizing module may not exist as a direct sum of injective hulls of simple modules. Generalizing the theory to non-commutative settings, particularly to artinian algebras and finite-dimensional algebras over fields, would be a significant advance.

\item {Connections to tilting theory:} There is growing evidence that Matlis reflexive modules play a role in the theory of tilting and cotilting modules over Noetherian rings. A precise characterization of the tilting modules that are Matlis reflexive, and the relationship between reflexivity and the existence of tilting classes, remains to be explored.

\item {Computational aspects:} Given the importance of Matlis reflexive modules in local cohomology and algebraic geometry, there is a need for effective algorithms to decide whether a given module (presented, for example, by generators and relations) is Matlis reflexive. The development of such algorithms would have practical applications in computer algebra systems.
\end{enumerate}

\end{document}